\documentclass[10pt]{article}
\usepackage [latin1]{inputenc}
\usepackage{amsmath, amssymb, amsthm, a4wide}
\usepackage{geometry}
\usepackage{fancyhdr}
\usepackage{enumerate}
\newtheorem{theorem}{Theorem}[section]

\newtheorem{definition}[theorem]{Definition}
\newtheorem{proposition}[theorem]{Proposition}
\newtheorem{example}[theorem]{Example}
\newtheorem{remark}[theorem]{Remark}

\numberwithin{equation}{section}

\allowdisplaybreaks[4]
\newcounter{newlist}

\newcounter{nnnewlist}

\newcounter{nelist}

\def\vp{\varphi}
\def\ve{\varepsilon}

\def\br{\mathbb{R}}

\def\cp{\mathcal{P}}
\def\cf{\mathcal{F}}
\def\cb{\mathcal{B}}

\def\cf{\mathcal{F}}

\def\lu{\underline{\mu}}
\def\ou{\overline{\mu}}

\def\be{\mathbb{E}}
\def\pe{\mathbb{E}^{\cp}}

\def\bn{\mathbb{N}}

\def\vp{\varphi}
\def\ve{\varepsilon}

\def\lt{\left}
\def\rt{\right}

\begin{document}

\title{On the necessary and sufficient conditions for Peng's law of large numbers  under sublinear expectations}

\author{Xinpeng LI\ \ \ Gaofeng ZONG\thanks{Partially supported by NSF of Shandong Province (No. ZR2021MA018), National Key R\&D Program of China (No. 2018YFA0703900), NSF of China (No.11601281, No.11501325), the Young Scholars Program of Shandong University, the China Postdoctoral Science Foundation (Grant No.2018T110706, No.2018M642641)).
\
Corresponding author:  Gaofeng Zong, Email: gf\b{ }zong@126.com}\\Reserach Center for Mathematics and Interdisciplinary Sciences\\Shandong University\\266237, Qingdao, China;\\
School of Mathematics and Quantitative Economics,\\ Shandong University of Finance and Economics,\\ 250014, Jinan, China
\\ $^\ast$Email: gf\b{ }zong@126.com}
\date{}

\maketitle

\noindent\textbf{Abstract.}
In this paper, we firstly establish the weak laws of large numbers on the canonical space $(\br^\bn,\cb(\br^\bn))$ by traditional truncation method and Chebyshev's inequality as in the classical probability theory. Then we extend them from the canonical space to the general sublinear expectation space. The necessary and sufficient conditions for Peng's law of large numbers are obtained.  



\section{Introduction}

The general law of large numbers (LLN) on sublinear expectation space $(\Omega,\mathcal{H},\be)$  was established by Peng \cite{pengsur,Peng2007b,Peng2019}, known as Peng's law of large numbers: Let $\{X_i\}_{i\in\bn}$ be independent and identically distributed (i.i.d.) random variables with
\begin{equation}\label{uc}
\lim_{\lambda\to\infty}\be[(|X_1|-\lambda)^+]=0.
\end{equation}
then
\begin{equation}
\label{wlln}
\lim_{n\to\infty}\be\lt[\vp\lt(\frac{1}{n}\sum_{i=1}^nX_i\rt)\rt]=\max_{-\be[-X_1]\leq\mu\leq\be[X_1]}\vp(\mu), \ \ \ \forall \vp\in C_b(\br).
\end{equation}
Here the notions of i.i.d. are initially introduced by Peng \cite{Peng2007b}, which are directly defined by the sublinear expectation $\be$ instead of the capacity (for example, see Marinacci \cite{massimo} and  Maccheroni and Marinacci \cite{mm}). The convergence in the form (\ref{wlln}) is similar to the weak convergence in the classical probability theory, as such (\ref{wlln}) can be regarded as the weak LLN. The limit distribution in the right hand of (\ref{wlln}) is called maximal distribution, see Peng \cite{Peng2019}. There are also strong LLNs on the sublinear expectation space, see Chen \cite{chen1}, Chen et al. \cite{cwl}, Zhang \cite{zhang,zhang2021}, Guo and Li \cite{GL}, etc.

This paper focuses on the weak LLN on the sublinear expectation space, which can not be implied directly by the strong one. Recently, the convergence rate of Peng's LLN was studied under higher moment conditions, see Fang et al. \cite{FPSS}, Song \cite{song} and Hu et al. \cite{hll}. But is  (\ref{uc}) a sharpest condition to obtain (\ref{wlln})? We recall that the classical weak LLN shows that, for i.i.d. sequence $\{X_i\}$ on the probability space $(\Omega,\cf,P)$, $\frac{\sum_{i=1}^nX_i}{n}$ converges to $\mu$ in probability if and only if
\begin{align}\label{c1}
\lim_{n\rightarrow\infty}nP(|X_1|\geq n)=0\ \ \text{and}  \ \ \lim_{n\rightarrow\infty}E_P[(-n\vee X_1)\wedge n]=\mu,
\end{align}
which inspires us to explore the necessary and sufficient conditions for Peng's LLN analogous to (\ref{c1}).

Different from most literature on the limit theorems under sublinear expectation, we adopt the formulation similar to Li and Peng \cite{li2021}. We firstly establish LLN on the canonical space $(\br^\bn,\cb(\br^\bn))$ by traditional truncation method and Chebyshev's inequality. Then we extend the LLN to the general sublinear expectation space introduce by Peng \cite{Peng2019}. The necessary and sufficient conditions for Peng's LLN are obtained,  which is a natural generalization of the classical LLN to the sublinear expectation theory. But unlike the classical situation, there is a counterexample shows that the weak LLN (\ref{wlln}) does not hold if $\be[|X_1|]<\infty$.

The rest of this paper is organized as follows. Section 2 describes the basic model on the canonical space.  We establish the LLNs on the canonical space in Section 3 and then extend them to the general sublinear expectation space in Section 4. The proofs are given in Section 5.

\section{Basic model on canonical space}

Following Li and Peng \cite{li2021}, let $(\br^\bn,\cb(\br^\bn))$ be the canonical space and $\{X_i\}_{i\in\bn}$ be the sequence of canonical random variables defined by $X_i(\omega)=\omega_i$ for $\omega=(\omega_1,\cdots,\omega_n,\cdots)\in\br^\bn$. For each $i\in\bn$, $\cp_i$ is the convex and weakly compact set of probability measures on $(\br,\mathcal{B}(\br))$ characterizing the uncertainty of the distributions of $X_i$ as described in Peng \cite{Peng2019}. The set of joint laws on $(\br^\bn,\mathcal{B}(\br^\bn))$ via the probability kernels are defined by
\begin{equation}
\begin{array}
[c]{l}
\mathcal{P}=\{P:P(A^{(n)}\times\br^{\bn-n})=\int_\br\mu_1(dx_1)\int_{\br}\kappa_{2}(x_1,dx_2)\cdots\int_\br I_{A^{(n)}}\kappa_n(x_1,\cdots,x_{n-1},dx_{n})\\
\forall n\geq 1, \ \forall A^{(n)}\in\cb(\br^n),\ \ \kappa_i(x_1,\cdots,x_{i-1},\cdot)\in\cp_i,\ \ 2\leq i\leq n,\ \mu_1\in\cp_1\},
\end{array}
\label{e1}%
\end{equation}
where for each $i\geq 2$, $\kappa_i(x_1,\cdots,x_{i-1}, dx_i)$ is the probability kernel satisfying:
\begin{itemize}
\item[(i)] $\forall (x_1,\cdots,x_{i-1})\in\br^{i-1}$, $\kappa_i(x_1,\cdots,x_{i-1},\cdot)$ is a probability measure on $(\br,\mathcal{B}(\br))$.
\item[(ii)] $\forall B\in\mathcal{B}(\br)$, $\kappa_i(\cdot, B)$ is $\mathcal{B}(\br^{i-1})$-measurable.
\end{itemize}
The existence of $P$ on $(\br^\bn,\cb(\br^\bn))$ is shown by Ionescu-Tulcea theorem.

Since $\cp_i$ is independent of $(x_1,\cdots,x_{i-1})$, we call that $X_i$ is independent of $(X_1,\cdots,X_{i-1})$. Here independence means the uncertainty of the distributions of $X_i$ is independent of random vector $(X_1,\cdots, X_{i-1})$. It is obvious that such independence is not symmetric in general. If we further assume that the uncertainties of each $X_i$ are the same, i.e., $\cp_i=\cp_1$, then $\{X_i\}_{i\in\bn}$ is an i.i.d. sequence. In this case, $\cp$ is determined by the set $\cp_1$ of the distributions of $X_1$.

For each $\mathcal{B}(\br^\bn)$-measurable random variable $X$, the sublinear expectation $\pe$ is defined by
$$\mathbb{E}^{\cp}[X]:=\sup_{P\in\cp}E_P[X].$$
The corresponding  capacities $v$ and $V$ are defined respectively
$$v(A)=\inf_{P\in\cp}P(A), \ \ \ V(A)=\sup_{P\in\cp}P(A), \ \ \forall A\in\cb(\br^\bn).$$
The canonical filtration $\{\cf_i\}_{i\in\bn}$ is defined as $\cf_i=\sigma(X_k, 1\leq k\leq i)$ with convention $\cf_0=\{\emptyset,\br^\bn\}$.

By the regularity of conditional probability, we have
$$E_P[\vp(X_i)|\mathcal{F}_{i-1}](\omega)=E_{\kappa_i(x_1,\cdots,x_{i-1},\cdot)}[\vp(X_i)],\ \ \ P-\text{a.s.},$$
where $\omega=(x_1,\cdots,x_n,\cdots)\in\br^\bn$ and $\kappa_i$ is the probability kernel introduced by $P$.

We immediately have the following proposition since  $\kappa_i(x_1,\cdots,x_{i-1},\cdot)\in\cp_i$, $1\leq i\leq n$.
\begin{proposition}\label{prop23}
For each $P\in\cp$ and $\vp\in C_{b}(\br)$, we have
\begin{equation}-\mathbb{E}^{\cp}[-\vp(X_i)]\leq E_P[\vp(X_i)|\cf_{i-1}]\leq\mathbb{E}^{\cp}[\vp(X_i)], \ \ P-\text{a.s.} \ i\in\bn.\label{eq22}\end{equation}
\end{proposition}

\begin{remark}
The inequality (\ref{eq22}) also holds on the general sublinear expectation space, see Guo and Li \cite{GL}.
\end{remark}

In the end of this section, we generalize Ottaviani's inequality to the canonical space.
\begin{proposition}\label{otta}
Let $c$ be a real number satisfying $0<c<1$  and $S_n=\sum_{i=1}^nX_i$. If there exist constants $\alpha_n\in\br$ such that
$$\max_{1\leq k\leq n} V(|S_n-S_k|\geq\alpha_n)\leq c,$$
then we have
$$V\lt(\max_{1\leq k\leq n}|S_k|\geq 2\alpha_n\rt)\leq \frac{1}{1-c}V(|S_n|\geq \alpha_n).$$
\end{proposition}
\begin{remark}
The general Levy maximal inequality on the sublinear expectation space was firstly proved in Zhang \cite{zhang2019}, but additional $\ve$-condition is required.
\end{remark}

\section{Law of large numbers on the canonical space}

In this section, we establish the weak LLN on the canonical space by traditional truncation method and Chebyshev¡¯s inequality as in the classical probability theory.

Given a weakly compact and convex set $\cp_0$ of probability measures on $(\br,\cb(\br))$, for fixed $n\in\bn$, we can construct $\cp$ on $(\br^\bn,\cb(\br^\bn))$ by (\ref{e1}) with $\cp_i=\cp_0$ for $i\in\bn$, such that the sequence of canonical random variables $\{X_i\}_{i\in\bn}$ are i.i.d., where $X_i(\omega)=\omega_i$, $\forall \omega=(\omega_1,\cdots,\omega_n,\cdots)\in\br^\bn$. The sublinear expectation $\be^{\cp}$ and the capacities $v$ and $V$ corresponding to $\cp$ are defined as in the previous section.

Inspired by the weak LLN under sublinear expectation without the moment condition in Chen et al. \cite{chen}, we also have the similar LLN on the canonical space.

\begin{theorem}\label{th31}
Let $\{X_i\}_{i\in\bn}$ be a canonical i.i.d. sequence under $\cp$ which is constructed by (\ref{e1}), and we further assume that
$$\lim_{n\to\infty}nV(|X_1|\geq n)=0.$$
Then for any $\ve>0$,
\begin{equation}\label{llnv}
\lim_{n\to\infty}v\lt(\underline{\mu}_n-\ve \leq \frac{\sum_{i=1}^nX_i}{n} \leq \overline{\mu}_n+\ve\rt)=1,
\end{equation}
where $\underline{\mu}_n=-\be^{\cp}[(-n\vee -X_1)\wedge n]$ and $\overline{\mu}_n=\be^{\cp}[(-n\vee X_1)\wedge n]$.

Furthermore, if $\{\underline{\mu}_n\}_{n\in\bn}$ and $\{\overline{\mu}_n\}_{n\in\bn}$ are bounded, then
\begin{equation}\label{llnp}
\lim_{n\to\infty}\lt|\be^{\cp}\lt[\vp\lt(\frac{\sum_{i=1}^nX_i}{n}\rt)\rt]-\max_{\underline{\mu}_n\leq\mu\leq\overline{\mu}_n}\vp(\mu)\rt|=0, \ \ \forall \vp\in C_b(\br).
\end{equation}

\end{theorem}

Now we establish Peng's LLN on the canonical space with the necessary and sufficient conditions.
\begin{theorem}\label{plln}
Let $\{X_i\}_{i\in\bn}$ be a canonical i.i.d. sequence under $\cp$ which is constructed by (\ref{e1}). Then
\begin{equation}\label{e33}
\lim_{n\to\infty}\be^{\cp}\lt[\vp\lt(\frac{\sum_{i=1}^nX_i}{n}\rt)\rt]=\max_{\underline{\mu}\leq\mu\leq\overline{\mu}}\vp(\mu), \ \ \forall \vp\in C_b(\br),
\end{equation}
where $-\infty<\lu\leq\ou<+\infty$, if and only if the following three conditions hold:
\begin{itemize}
\item[(i)] $\lim_{n\to\infty}nV(|X_1|\geq n)=0$.
\item[(ii)] $\lim_{n\to\infty}\be^{\cp}[(-n\vee X_1)\wedge n]=\ou$.
\item[(iii)] $\lim_{n\to\infty}-\be^{\cp}[(-n\vee -X_1)\wedge n]=\lu$.
\end{itemize}
\end{theorem}

\section{Law of large numbers on the sublinear expectation space}

In this section, we extend the obtained weak LLN to the general sublinear expectation space $(\Omega,\mathcal{H},\be)$ introduced by Peng \cite{Peng2019}.

Let $\Omega$ be a given set and let $\mathcal{H}$ be a linear space of real
functions defined on $\Omega$ such that $c\in\mathcal{H}$ for $c\in\br$, and if $X\in\mathcal{H}$, then $|X|\in\mathcal{H}$. We further assume that
$X_{1},\ldots,X_{n}\in \mathcal{H}$, then $\varphi(X_{1},\cdots,X_{n}%
)\in \mathcal{H}$ for each $\varphi \in C_{b.Lip}(\mathbb{R}^{n})$, where
$C_{b.Lip}(\mathbb{R}^{n})$ denotes the space of all bounded and  Lipschitz
functions on $\mathbb{R}^{n}$. $\mathcal{H}$ is considered as the space of
random variables.

\begin{definition}
A sublinear expectation $\be$ on $\mathcal{H}$ is a functional $\be:\mathcal{H}\rightarrow \mathbb{R}$ satisfying the following
properties: for all $X,Y\in \mathcal{H}$, we have
\begin{description}
\item[(a)] Monotonicity: $\be[X]\geq\be[Y]$ if
$X\geq Y$.

\item[(b)] Constant preserving: $\be[c]=c$ for $c\in \mathbb{R}$.

\item[(c)] Sub-additivity: $\be[X+Y]\leq \be[X]+\be[Y]$.

\item[(d)] Positive homogeneity: $\be[\lambda X]=\lambda\be[X]$ for $\lambda \geq0$.
\end{description}

The triple $(\Omega,\mathcal{H},\be)$ is called a sublinear
expectation space.
\end{definition}

Let $X=(X_1,\cdots,X_n)$ be a given $n$-dimensional random vector on a sublinear expectation space $(\Omega,\mathcal{H},\be)$. We define a functional on $C_{b.Lip}(\br^n)$ by
$$\mathbb{{F}}_X[\vp]:=\be[\vp(X)], \ \ \forall \vp\in C_{b.Lip}(\br^n).$$
The triple $(\br^n, C_{b.Lip}(\br^n),\mathbb{\hat{F}}_X[\cdot])$ forms a sublinear expectation space, and $\mathbb{\hat{F}}_X$ is called the sublinear distribution of $X$.

\begin{definition}
Let $X$ and $Y$ be two random variables on $(\Omega,\mathcal{H},\be)$. $X$ and $Y$ are called identically distributed, denoted by
$X\overset{d}{=}Y$, if for each $\varphi \in C_{b.Lip}(\mathbb{R})$,
\[
\be[\varphi(X)]=\be[\varphi(Y)].
\]

\end{definition}

\begin{definition}
\label{new-de1}Let $\left \{  X_{i}\right \}  _{i\in\bn}$ be a sequence of
random variables on $(\Omega,\mathcal{H},\be)$. $X_{i}$ is said
to be independent of $\left(  X_{1},\ldots,X_{i-1}\right)  $ under
$\be$, if for each $\vp\in C_{b.Lip}(\br^{n})$,
$$
\be\left[\varphi\left(X_{1}, \cdots, X_{i}\right)\right]=\be\left[\left.\be\left[\varphi\left(x_{1}, \cdots, x_{i-1},X_i\right)\right]\right|_{\left(x_{1}, \cdots, x_{i-1}\right)=\left(X_{1}, \cdots, X_{i-1}\right)}\right].
$$
The sequence of random variables $\left \{  X_{i}\right \}  _{i\in\bn}$ is
said to be independent, if $X_{i+1}$ is independent of $\left(  X_{1}%
,\ldots,X_{i}\right)  $ for each $i\geq1$.
\end{definition}

If the sublinear distribution of $X_1$ is given, then there exists a unique weakly compact and convex set of probability measures $\cp_0$ on $(\br,\cb(\br))$ such that
\begin{equation}\label{phi}
\be[\vp(X_1)]=\mathbb{E}^{\cp_0}[\vp], \ \forall \vp\in C_{b.Lip}(\br),
\end{equation}
where the existence was proved in Hu and Li \cite{hl} and the uniqueness was proved in Li and Lin \cite{li}.

The relation between the i.i.d. sequence on sublinear expectation space and the one on canonical space is characterized by the following proposition, which was proved in Li and Peng \cite{li2021}.
\begin{proposition}\label{p44}
Let $\{X_i\}_{i\in\bn}$ be an i.i.d. sequence on sublinear expectation space $(\Omega,\mathcal{H},\be)$ and the sublinear distribution of $X_1$ can be represented by $\cp_0$ as in (\ref{phi}). Then there exists an sequence of i.i.d. canonical random variables  $\{Y_i\}_{i\in\bn}$ on canonical space $(\br^\bn,\cb(\br^\bn))$ with the set of probability measures $\cp$ constructed by (\ref{e1}) with $\cp_i=\cp_0$ such that for each $n\in\bn$ and $\vp\in C_{b.Lip}(\br^n)$,
$$\be[\vp(X_1,\cdots,X_n)]=\mathbb{E}^{\cp}[\vp(Y_1,\cdots,Y_n)].$$
\end{proposition}

Now we extend Peng's LLN (Theorem \ref{plln}) from canonical space to sublinear expectation space. We need to characterize  condition (i) of Theorem \ref{plln} directly by sublinear expectation $\be$ instead of capacity $V$. For each $n\in\bn$, we define $\psi_n\in C_{b.Lip}(\br)$ as
$$\psi_n(x)=\sup_{y\in\br}\{n\mathbf{1}_{\{|y|\geq n\}}-n|y-x|\}.$$
\begin{proposition}\label{cvp}
For each $n\in\bn$, we have
\begin{equation}\label{ee42}
n\mathbf{1}_{\{|x|\geq n\}}\leq\psi_n(x)\leq n\mathbf{1}_{\{|x|\geq n-1\}}.
\end{equation}
\end{proposition}

\begin{theorem}\label{th46}
Let $\{X_i\}$ be an i.i.d. sequence on $(\Omega,\mathcal{H},\be)$. Then
\begin{equation}
\label{ee41}\lim_{n\rightarrow\infty}\be\lt[\vp\lt(\frac{\sum_{i=1}^nX_i}{n}\rt)\rt]=\max_{\lu\leq\mu\leq\ou}\vp(\mu), \ \ \forall \vp\in C_b(\br).
\end{equation}
if and only if the following three conditions hold:
\begin{itemize}
\item[(i)] $\lim_{n\rightarrow\infty}\be[\psi_n(X_1)]=0.$
\item[(ii)] $\lim_{n\rightarrow\infty}\be[(-n\vee X_1)\wedge n]=\ou$.
\item[(iii)] $\lim_{n\rightarrow\infty}-\be[(-n\vee -X_1)\wedge n]=\lu$.
\end{itemize}
\end{theorem}

In the end of this section, we provide some examples. The first one shows that there exists a random variable $X$ such that $\lim_{n\to\infty}\be[\psi_n(X)]=0$ but condition (\ref{uc}) does not hold, i.e., $\lim_{\lambda\to\infty}\be[(|X|-\lambda)^+]\neq0$.

\begin{example}
\label{Exm3}Let $\Omega=\mathbb{N}$, $\mathcal{P}=\{P_{n}:n\in
\mathbb{N}\}$ where $P_{1}(\{1\})=1$ and
$P_{n}(\{1\})=1-\frac{1}{n^{2}}$,  $P_{n}(\{kn\})=\frac{1}{n^{3}}$,
$k=1,2,\ldots,n$, for $n=2,3,\cdots$. We consider a function $X$ on $\mathbb{N}$
defined by $X(n)=n$, $n\in \mathbb{N}$. We can calculate that $\lim_{n\to\infty}\be^{\cp}[\psi_n(X)]=0$ but
$\lim_{\lambda\to\infty}\mathbb{E}^{\cp}[(|X|-\lambda)^+]=\frac
{1}{2}\neq0$.
\end{example}

The next example shows that the Peng's LLN may fail only with the first moment condition, which is different from the classical LLN.

\begin{example}
Let $\Omega=\bn$, $\cp=\{P_n:n\in\mathbb{N}\}$, where $P_n(\{0\})=1-\frac{1}{n}$ and $P_n(\{n\})=\frac{1}{n}$ for $n=1,\cdots,$. Consider a function $X$ on $\mathbb{N}$ defined by $X(n)=n$, $n\in\mathbb{N}$. The sublinear expectation $\be$ is defined by $\be[\cdot]=\sup_{P\in\cp}E_P[\cdot]$. It is clear that $\be[X]=-\be[-X]=1$.  By Peng \cite{Peng2019}, we can construct i.i.d sequence $\{X_i\}_{i\in\bn}$ under $\be$ such that $X_i$ have the same distribution with $X$. But the weak law of large numbers does not hold, i.e., we can find $\vp\in C_b(\br)$, such that
$$\lim_{n\to\infty}\be\lt[\vp\lt(\frac{S_n}{n}\rt)\rt]\neq\vp(1).$$
\end{example}
Indeed, we consider $\vp(x)=1\wedge(1-x)^+$. In this case, by simple calculation, we can obtain $\be[\vp(\frac{x+X}{n})]=\vp(\frac{x}{n}).$ Then we have
\begin{align*}
  \be\lt[\vp\lt(\frac{X_1+\cdots+X_n}{n}\rt)\rt]
  &=\be\lt[\be\lt[\vp\lt(\frac{x+X_n}{n}\rt)\rt]\lt|_{x=X_1+\cdots+X_{n-1}}\rt.\rt]\\
  &=\be\lt[\vp\lt(\frac{X_1+\cdots+X_{n-1}}{n}\rt)\rt]=\cdots=\be\lt[\vp\lt(\frac{X_1}{n}\rt)\rt]=\vp(0)=1.
\end{align*}
which implies that,
$$\lim_{n\rightarrow\infty}\be\lt[\vp\lt(\frac{X_1+\cdots+X_n}{n}\rt)\rt]=1\neq\vp(1)=0.$$

%
%
%
%
%
%

\section{Proofs}

We firstly prove Ottaviani's inequality.
\begin{proof}[Proof of Proposition \ref{otta}]
For $k=1,\cdots,n$, let
$$A_k=\lt\{\max_{l\leq k-1}|S_l|<2\alpha_n, |S_k|\geq 2\alpha_n\rt\},$$
Then
$$V\lt(\max_{1\leq k\leq n}|S_k|\geq 2\alpha_n\rt)=\be^\cp[\sum_{k=1}^n\mathbf{I}_{A_k}].$$
We have
\begin{align*}
\mathbf{1}_{\{|S_n|\geq \alpha_n\}}
&\geq \sum_{k=1}^n\mathbf{1}_{A_k}\mathbf{1}_{\{|S_n|\geq\alpha_n\}}
\geq\sum_{k=1}^n\mathbf{1}_{A_k}\mathbf{1}_{\{|S_n-S_k|<\alpha_n\}}\\
&=\sum_{k=1}^n\mathbf{1}_{A_k}(1-\mathbf{1}_{\{|S_n-S_k|\geq\alpha_n\}})
=(1-c)\sum_{k=1}^n\mathbf{1}_{A_k}+\sum_{k=1}^n\mathbf{1}_{A_k}(c-\mathbf{1}_{\{|S_n-S_k|\geq\alpha_n\}})
\end{align*}
which implies that
$$(1-c)\sum_{k=1}^n\mathbf{1}_{A_k}\leq\mathbf{1}_{\{|S_n|\geq\alpha_n\}}+\sum_{k=1}^n\mathbf{1}_{A_k}(\mathbf{1}_{\{|S_n-S_k|\geq \alpha_n\}}-c).$$
For each $P\in\cp$, by (\ref{e1}), we obtain
\begin{align*}
(1-c)P\lt(\max_{k\leq n}|S_k|\geq 2\alpha_n\rt)&\leq P(|S_n|\geq\alpha_n)+\sum_{k=1}^nP(A_k)(P(|S_n-S_k|\geq\alpha_n)-c)\\
&\leq P(|S_n|\geq\alpha_n)\leq V(|S_n|\geq\alpha_n).
\end{align*}
The proof is completed.
\end{proof}

Now we give the proofs of LLNs on canonical space.

\begin{proof}[Proof of Theorem \ref{th31}]
Let $\tilde{X}_i=(-n\vee X_i)\wedge n$, $1\leq i\leq n$, $S_n=\sum_{i=1}^nX_i$ and $T_n=\sum_{i=1}^n\tilde{X}_i$.

For each $P\in\cp$, Proposition \ref{prop23} implies that
$$-\be^{\cp}[-\tilde{X}_i]\leq E_P[\tilde{X}_i|\cf_{i-1}]\leq\be^{\cp}[\tilde{X}_i], \ \ \forall i\in\bn.$$

By Chebyshev's inequality, we have
\begin{align*}
P\lt(\frac{\sum_{i=1}^n\lt(\tilde{X}_i-E_P[\tilde{X}_i|\cf_{i-1}]\rt)}{n}>\ve\rt)&\leq\frac{\sum_{i=1}^nE_P\lt[\lt(\tilde{X}_i-E_P[\tilde{X}_i|\cf_{i-1}]\rt)^2\rt]}{n^2\ve^2}\leq\frac{2\sum_{i=1}^nE_P[|\tilde{X}_i|^2]}{n^2\ve^2}\\
&\leq \frac{2\sum_{i=1}^n\be^{\cp}[|\tilde{X}_i|^2]}{n^2\ve^2}=\frac{2\be^{\cp}[|\tilde{X}_1|^2]}{n\ve^2}.
\end{align*}
Noting that,
$$P(S_n\neq T_n)\leq\sum_{i=1}^nP(|X_i|\geq n)\leq \sum_{i=1}^nV(|X_i|\geq n)=nV(|X_1|\geq n),$$
thus
\begin{align*}
P\lt(\frac{S_n}{n}>\ou_n+\ve\rt)\leq& P(S_n\neq T_n)+P\lt(\frac{\sum_{i=1}^n\lt(\tilde{X}_i-E_P[\tilde{X}_i|\cf_{i-1}]\rt)}{n}>\frac{\ve}{2}\rt)\\
&+P\lt(\frac{\sum_{i=1}^nE_P[\tilde{X}_i|\cf_{i-1}]}{n}>\ou_n+\frac{\ve}{2}\rt)\\
\leq &nV(|X_1|\geq n)+\frac{8}{n\ve^2}\be^{\cp}[|\tilde{X}_1|^2]+P\lt(\frac{\sum_{i=1}^n\be^{\cp}[\tilde{X}_i]}{n}>\ou_n+\frac{\ve}{2}\rt)\\
=&nV(|X_1|\geq n)+\frac{8}{n\ve^2}\be^{\cp}[|\tilde{X}_1|^2]
\end{align*}
We claim that
$$\lim_{n\to\infty}\frac{1}{n}\be^{\cp}[|\tilde{X}_1|^2]=0.$$
Indeed,
\begin{align*}
\frac{1}{n}\be^{\cp}[|\tilde{X}_1|^2]&\leq\frac{1}{n}\sup_{P\in\cp}E_P[|X_1|^2\mathbf{1}_{\{|X_1|\leq n\}}]+nV(|X_1|\geq n)\\
&\leq\frac{2}{n}\sup_{P\in\cp}\int_0^ntP(|X_1|>t)dt+nV(|X_1|\geq n)\\
&\leq \frac{2}{n}\int_0^ntV(|X_1|>t)dt+nV(|X_1|\geq n)\to 0,
\end{align*}
since $nV(|X_1|\geq n)\to 0$.

Thus we imply
$$V\lt(\frac{S_n}{n}>\ou_n+\ve\rt)\rightarrow 0,$$
Similarly, we can prove that $V(\frac{S_n}{n}<\lu_n-\ve)\rightarrow 0$. Therefore, (\ref{llnv}) holds.

The proof of (\ref{llnp}) is divided into two steps.

Firstly, for each $\ve>0$, we have
$$\be^{\cp}\lt[\vp\lt(\frac{S_n}{n}\rt)\rt]\leq\max_{\lu_n-\ve\leq\mu\leq\ou_n+\ve}\vp(\mu)+\max_{\mu\in\br}\vp(\mu)\lt(V\lt(\frac{S_n}{n}>\ou_n+\ve\rt)+V\lt(\frac{S_n}{n}<\lu_n-\ve\rt)\rt).$$
Let $n\rightarrow\infty$, we obtain
$$\limsup_{n\rightarrow\infty}\lt\{\be^{\cp}\lt[\vp\lt(\frac{S_n}{n}\rt)\rt]-\max_{\lu_n-\ve\leq\mu\leq\ou_n+\ve}\vp(\mu)\rt\}\leq 0, \ \ \ \forall \vp\in C_{b}(\br).$$
Since $\vp$ is continuous and $\{\lu_n\}_{n\in\bn}$ and $\{\ou_n\}_{n\in\bn}$ are bounded, let $\ve\rightarrow 0$, we imply that,
\begin{equation}\label{eq1}
\limsup_{n\rightarrow\infty}\lt\{\be\lt[\vp\lt(\frac{S_n}{n}\rt)\rt]-\max_{\lu_n\leq\mu\leq\ou_n}\vp(\mu)\rt\}\leq 0, \ \ \ \forall \vp\in C_{b}(\br).
\end{equation}

Secondly, for fixed $\vp\in C_{b}(\br)$ with $|\vp(x)|\leq C_\vp$,  and for each $\tilde{\ve}>0$, there exists  $\mu_n\in(\lu_n,\ou_n)$ such that $\vp(\mu_n)\geq \max_{\lu_n\leq\mu\leq\ou_n}\vp(\mu)-\tilde{\ve}$. By the construction of $\cp$,  we can find $P_*\in\cp$ such that $\{X_i\}$ is an i.i.d. sequence under $P_*$ with $E_{P_*}[\tilde{X}_1]=\mu_n$. In this case, by the classical weak law of large numbers, for each $\ve>0$,
$$P_*\lt(\lt|\frac{S_n}{n}-\mu_n\rt|>\ve\rt)\rightarrow 0.$$
By the continuity of $\vp$, there exists $\delta>0$ such that $|\vp(x)-\vp(y)|<\ve$ provided $|x-y|<\delta$. Then
$$\lt|E_{P_*}\lt[\vp\lt(\frac{S_n}{n}\rt)\rt]-\vp(\mu_n)\rt|\leq \ve+2C_\vp P_*\lt(\lt|\frac{S_n}{n}-\mu_n\rt|>\delta\rt),$$
which implies that
$$\lim_{n\rightarrow\infty}\lt|E_{P_*}\lt[\vp\lt(\frac{S_n}{n}\rt)\rt]-\vp(\mu_n)\rt|=0.$$
Finally, for each $\tilde{\ve}>0$,

\begin{equation}\label{eq2}
\liminf_{n\rightarrow\infty}\lt\{\be^{\cp}\lt[\vp\lt(\frac{S_n}{n}\rt)\rt] - \max_{\lu_n \leq \mu \leq \ou_n} \vp(\mu)\rt\} \geq \liminf_{n\rightarrow\infty}\lt\{E_{P_*}\lt[\vp\lt(\frac{S_n}{n}\rt)\rt]-\vp(\mu_n)\rt\}-\tilde{\ve}=-\tilde{\ve},
\end{equation}
Combining (\ref{eq1}) and (\ref{eq2}), we can see that (\ref{llnp}) holds.
\end{proof}

Now we prove the sufficient and necessary conditions for Peng's LLN on the canonical space.

\begin{proof}[Proof of Theorem \ref{plln}]
If (i)-(iii) holds, then by Theorem \ref{th31}, (\ref{e33}) holds. Conversely, we assume that (\ref{e33}) holds.

Let $\hat{\mu}=2[\ou]+2>2\ou$, which is an integer. By (\ref{e33}), there exists $N$ such that
$$V\lt(|S_n|\geq \frac{n\hat{\mu}}{2}\rt)\leq \frac{1}{4},\ \  \forall n>N.$$
Then for $N\leq k\leq n$, we have
\begin{align*}
V(|S_k-S_n|\geq n\hat{\mu})&\leq V\lt(|S_k|\geq \frac{n\hat{\mu}}{2}\rt)+V\lt(|S_n|\geq \frac{n\hat{\mu}}{2}\rt)\\
&\leq V\lt(|S_k|\geq \frac{k\hat{\mu}}{2}\rt)+V\lt(|S_n|\geq \frac{n\hat{\mu}}{2}\rt)\leq\frac{1}{2}.
\end{align*}
Applying Ottaviani's inequality, we obtain
$$V(\max_{N\leq k\leq n}|S_k|\geq  2n\hat{\mu})\leq 2V(|S_n|\geq  n\hat{\mu}),$$
which implies that
\begin{align*}
V(\max_{N+1\leq k\leq n}|X_k|\geq  4n\hat{\mu})
&\leq V(\max_{N+1\leq k\leq n}|S_k|\geq  2n\hat{\mu})+V(\max_{N\leq k\leq n-1}|S_k|\geq  2n\hat{\mu})\\
&\leq 4V(|S_n|\geq  n\hat{\mu}).
\end{align*}
On the other hand, by the construction of $P$ in (\ref{e1}),
\begin{align*}
V(\max_{N+1\leq k\leq n}|X_k|\geq 4n\hat{\mu})
&=1-v(\max_{N+1\leq k\leq n}|X_k|< 4n\hat{\mu})\\
&=1-\Pi_{k=N+1}^nv(|X_k|< 4n\hat{\mu})\\
&=1-\lt(1-V(|X_1|\geq 4n\hat{\mu})\rt)^{n-N}\\
&\geq1-e^{-(n-N)V(|X_1|\geq 4n\hat{\mu})}.
\end{align*}
Thus
$$(n-N)V(|X_1|\geq 4n\hat{\mu})\leq -\ln(1-4V(|S_n|\geq n\hat{\mu})).$$
Since $V\lt(\frac{|S_n|}{n}\geq \hat{\mu}\rt)\to 0$, as $n\to\infty$, if $n>2N$, we have
$$\frac{n}{2}V(|X^{n}_1|\geq 4n\hat{\mu})\leq (n-N)V(|X_1|\geq 4n\hat{\mu})\to 0,$$
which implies that
$$\lim_{n\to\infty}nV(|X_1|\geq n)=0.$$
Then by Theorem \ref{th31}, we can deduce that
$$\lim_{n\to\infty}v(\lu_n-\ve\leq\frac{S_n}{n}\leq\ou_n+\ve)=1, \ \ \forall \ve>0,$$
where $\ou_n=\be^{\cp}[(-n\vee X_1)\wedge n]$ and $\lu_n=-\be^{\cp}[(-n\vee -X_1)\wedge n]$.
Thus
$$\lim_{n\to\infty}\lt|\be\lt[\vp\lt(\frac{S_n}{n}\rt)\rt]-\max_{\lu_n\leq\mu\leq\ou_n}\vp(\mu)\rt|=0.$$
Combining with (\ref{e33}), we can see that (ii) and (iii) hold.
\end{proof}

The proof of Proposition \ref{cvp} is just by the simple calculations.
\begin{proof}[Proof of Proposition \ref{cvp}] Taking $y=x$, we immediately obtain
$$n\mathbf{1}_{\{|x|\geq n\}}\leq\psi_n(x).$$
We also have
\begin{align*}
\psi_n(x)=&\max\lt\{\sup_{|y-x|\leq 1}\{n\mathbf{1}_{\{|y|\geq n\}}-n|y-x|\}, \sup_{|y-x|>1}\{n\mathbf{1}_{\{|y|\geq n\}}-n|y-x|\}\rt\}\\
\leq&\max\{n\mathbf{1}_{\{|x|\geq n-1\}}, 0\}=n\mathbf{1}_{\{|x|\geq n-1\}}.
\end{align*}
\end{proof}

In the end, we prove the Peng's LLN on sublinear expectation space with necessary and sufficient conditions.
\begin{proof}[Proof of Theorem \ref{th46}]
By Proposition \ref{p44}, for each $n\in\bn$, there exists a sequence of i.i.d. canonical random variables $\{Y_i\}$ on canonical space $(\br^\bn,\cb(\br^\bn))$ endowed with the set of probability measures $\cp$ constructed by (\ref{e1}) with $\cp_i=\cp_0$ such that
$$\be\lt[\vp\lt(\frac{\sum_{i=1}^nX_i}{n}\rt)\rt]=\be^{\cp}\lt[\vp\lt(\frac{\sum_{i=1}^nY_i}{n}\rt)\rt],\ \ \forall \vp\in C_{b}(\br).$$
Then by Theorem \ref{plln}, (\ref{ee41}) is equivalent to (i)-(iii) in Theorem \ref{plln}.

It is clear that (ii) and (iii) in Theorem \ref{th46} is equivalent to (ii) and (iii) in Theorem \ref{plln} respectively.

For (i) in Theorem \ref{th46}, it is equivalent to
$$\lim_{n\to\infty}\be^{\cp}[\psi_n(X_1)]=0,$$
which implies that
$$nV(|X_1|\geq n)\to 0.$$
Conversely, by (\ref{ee42}),
$$\be^{\cp}[\psi_n(X_1)]\leq nV(|X_1|\geq n-1)=\frac{n}{n-1}\lt((n-1)V(|X_1|\geq n-1)\rt)\to 0.$$
Thus (i) in Theorem \ref{plln} is equivalent to (i) in Theorem \ref{th46}. We complete the proof.
\end{proof}


\end{document}